\documentclass{amsart}
\usepackage{mathrsfs}
\usepackage{color,cleveref,xypic,MnSymbol,verbatim}
\usepackage{tikz-cd}
\tikzset{>=latex}
\usepackage{bbold}
\xyoption{all}
\title{The topological Petersson product.}
\author{L. Candelori}
\author{A. Salch}
\renewcommand{\smash}{\wedge}
\newtheorem{prop}{Proposition}[section]
\newtheorem{definition}[prop]{Definition}

\newtheorem{definition-proposition}[prop]{Definition-Proposition}

\newtheorem{theorem}[prop]{Theorem}
\newtheorem{corollary}[prop]{Corollary}

\theoremstyle{definition}
\newtheorem{example}[prop]{Example}
\newtheorem{examples}[prop]{Examples}

\DeclareMathOperator{\pt}{{\rm pt.}}

\DeclareMathOperator{\tmf}{{\it tmf}}

\DeclareMathOperator{\wt}{{\rm wt}}
\DeclareMathOperator{\petersson}{{\rm Petersson}}
\DeclareMathOperator{\tcf}{{\it tcf}}

\DeclareMathOperator{\Gal}{{\rm Gal}}

\newcommand{\ZZ}{\mathbb{Z}}

\begin{document}

\begin{abstract}
The nondegeneracy of the Petersson inner product on cusp forms, and the fact that Hecke operators are self-adjoint with respect to the Petersson product, together imply that the cusp forms have a basis consisting of Hecke eigenforms. In the literature on topological modular forms, no topological analogue of the Petersson product is to be found, and it is not known which topological spaces have the property that their topological cusp forms admit a basis consisting of eigenforms for the action of Baker's topological Hecke operators. In this note we define and study a natural topological Petersson product on complexified topological cusp forms, whose value on a one-point space recovers the classical Petersson product. We find that the topological Petersson product is usually degenerate: in particular, if $X$ is a space with nontrivial rational homology in any positive degree, then the Petersson product on complexified $tcf^*(X)$ is degenerate. Despite $tcf$-cohomology being a stable invariant, the topological Petersson product is an unstable invariant, vanishing on all suspensions. Nevertheless, we demonstrate nontriviality of the topological Petersson product by giving an explicit calculation of the topological Petersson product on complexified $tcf$-cohomology of the complex projective plane. We show that, for a compact Kahler manifold $X$, the Petersson product is nondegenerate on complexified $\tcf$-cohomology of $X$ in a range of Atiyah-Hirzebruch filtrations (essentially one-third of the possibly-nonzero Atiyah-Hirzebruch filtrations in the $tcf$-cohomology of $X$).
\end{abstract}

\maketitle
\tableofcontents

\section{Introduction}

The classical Petersson product of level $1$ cusp forms is the positive-definite Hermitian form given by the formula
\[ 
\langle f,g\rangle := \int_{\mathbb{H}/\mathrm{SL}_2(\ZZ)} f\bar{g} v^{k-2} du\,dv, \quad \tau = u + iv \in \mathbb{H} .\]
Each Hecke operator $T_m$ is self-adjoint with respect to the Petersson product on $S_k\otimes_{\mathbb{Z}}\mathbb{C}$. Consequently the action of each Hecke operator on $S_k$ is diagonalizable, and $S_k$ has a basis consisting of Hecke eigenforms; this is a fundamental, important result, since (among other reasons) the modular forms with $L$-functions expected to occur from irreducible degree $2$ Galois representations are the cusp forms, and for a normalized Hecke eigenform $f$ the $p$-th $L$-series coefficient is simply the eigenvalue of the action of $T_p$ on $f$. Having a basis for weight $k$ cusp forms consisting of Hecke eigenforms gives a powerful handle on their $L$-functions: we write a cusp form as a linear combination of Hecke eigenforms, and then on each summand we can recover the $L$-series coefficients from the Hecke eigenvalues.

It is quite satisfying to be able to get this fundamental result as consequences of the existence of a piece of {\em structure}, namely, an inner product on cusp forms with respect to which the Hecke operators are self-adjoint. One would like to see this same structure occur in the context of {\em topological} modular forms, with some similar consequences, like diagonalizability of the action of the Hecke operations on topological cusp forms. 

We define, in the most direct and na\"{i}ve possible way, a ``weight $j$ topological Petersson product'' 
\begin{align*} 
 \tcf^{-2k}(X) \otimes_{\mathbb{Z}} \tcf^{-2k}(X) 
  &\rightarrow H^{2(j-k)}(X; \mathbb{C})\end{align*}
for each unpointed finite CW-complex $X$.
When $X$ is a single point, this recovers the classical Petersson product in the case $j=k$, and is zero when $j\neq k$.
When $X$ is the complex projective plane, the weight $j$ topological Petersson product on $\tcf^{-2k}(X)$ vanishes unless $j=k$ or $j=k+1$, and recovers the classical Petersson product when $j=k$ or $j=k+1$. These examples are given in Examples \ref{product examples}.

Base-changing to $\mathbb{C}$, the topological Petersson product yields an $\mathbb{R}$-bilinear form
\begin{align*} 
 \left(\tcf^{-2k}(X) \otimes_{\mathbb{Z}}\mathbb{C}\right)\otimes_{\mathbb{R}} \left( \tcf^{-2k}(X)  \otimes_{\mathbb{Z}}\mathbb{C}\right) 
  &\rightarrow H^{2(j-k)}(X; \mathbb{C})\end{align*}
which is explicitly calculated in Proposition \ref{main prop} up to the Atiyah-Hirzebruch filtration, i.e., we calculate the topological Petersson product on the associated graded of the Atiyah-Hirzebruch filtration on $\tcf^{-2k}(X)\otimes_{\mathbb{Z}}\mathbb{C}$, where it depends only on the cup product on $H^*(X;\mathbb{Q})$ and the classical Petersson product.

We compute examples of topological Petersson products in Examples \ref{product examples}:
\begin{itemize}
\item On a single point, the topological Petersson product coincides with the classical Petersson product of cusp forms.
\item On a sphere of positive dimension, the topological Petersson product again coincides with the classical Petersson product of cusp forms, supported entirely on the basepoint. That is, the topological Petersson product is completely insensitive to the topology of the sphere.
\item However, the topological Petersson product over a space $X$ {\em does} detect some of the topology on $X$: we show that the topological Petersson product of topological cusp forms on $\mathbb{C}P^2$ agrees with the classical Petersson product on the basepoint, and it agrees with a weight-shifted copy of the classical Petersson product on the $2$-cell $S^2\subseteq \mathbb{C}P^2$. It is trivial on the topological cusp forms supported on the $4$-cell in $\mathbb{C}P^2$, however.
\end{itemize}
We give formula \eqref{product formula 200}, which provides an explicit and highly usable way to calculate the topological Petersson product of topological cusp forms over a space $X$, at least up to terms of higher Atiyah-Hirzebruch filtration. That is, formula \eqref{product formula 200} is an explicit formula for the topological Petersson product on the {\em associated graded} of the Atiyah-Hirzebruch (i.e., skeletal) filtration of $tcf^*(X)\otimes_{\mathbb{Z}}\mathbb{C}$.
One consequence is that the topological Petersson product is highly sensitive to the cup product structure of the cohomology of a space, and is consequently {\em not} a stable invariant of the space.

As another consequence, we get Corollary \ref{degeneracy of the petersson product}, which shows that, unfortunately, the topological Petersson product is degenerate on $\tcf^*(X)\otimes_{\mathbb{Z}}\mathbb{C}$ for finite CW-complexes $X$ with any nontrivial rational cohomology in positive degrees at all.

Here is the moral: in this note, we have carried out the most na\"{i}ve kind of investigation into what a Petersson product on topological modular forms might be, by simply considering the rational case. Rationalizing dramatically simplifies stable homotopy types, so one might expect this to trivialize all questions about topological Petersson products. Yet we see that the most basic and important question about the Petersson product---{\em is it degenerate, so that we can use it to show that the topological Hecke operators are diagonalizable, like how the argument goes for classical cusp forms?}---remains quite nontrivial even after our simplifications, and indeed, over any finite CW-complex with any nontrivial rational cohomology in positive degrees, the topological Petersson product is degenerate. Evidently some more involved techniques would be required to establish diagonalizability of topological Hecke operators on elliptic (co)homology.

Nevertheless, nondegeneracy of the topological Petersson product is not an entirely lost cause. We end on a small positive note: in Theorem \ref{compact kahler mfld thm}, we show that when $X$ is a compact K\"{a}hler manifold of complex dimension $d$, the topological Petersson product of $X$ is nondegenerate when restricted to the elements of $\tcf^*(X)\otimes_{\mathbb{Z}}\mathbb{C}$ of even Atiyah-Hirzebruch filtration $\leq 2d/3$. (When the rational homology of $X$ is concentrated in even degrees, then all elements of $\tcf^*(X)\otimes_{\mathbb{Z}}\mathbb{C}$ have even Atiyah-Hirzebruch filtration.)

Throughout this note, on various occasions, we have an $H\mathbb{C}$-module spectrum $X$ and a map of $H\mathbb{C}$-module spectra of the form 
\begin{equation}\label{bilinear form 1} X\smash_{H\mathbb{R}} X \rightarrow H\mathbb{C}.\end{equation} Naturally, this represents an $\mathbb{R}$-bilinear form on the spectrum $X$. If one is willing to keep track of a $\Gal(\mathbb{C}/\mathbb{R})$-action on the relevant spectra, then one could consider a map of the form 
\begin{equation}\label{bilinear form 2} X\smash_{H\mathbb{C}} \overline{X}\rightarrow H\mathbb{C},\end{equation}
i.e., a sesquilinear form on $X$. We suspect that, using equivariant spectra to keep track of the Galois action, the methods we pursue in this note in order to get an $\mathbb{R}$-bilinear Petersson product on $\tcf$-cohomology ought to also yield a sesquilinear Petersson product on $\tcf$-cohomology. However, bringing in ideas from equivariant stable homotopy would significantly increase the length and/or expected background knowledge of this note, while making only a modest improvement in the strength of the results, so we elected to consider only an $\mathbb{R}$-bilinear (rather than sesquilinear) Petersson product here.

We are grateful to Theo Johnson-Freyd for a helpful conversation early on during this project.

\section{Topological cusp forms.}
In this section we recall the construction of topological cusp forms. We learned about this construction from Theo Johnson-Freyd. Our understanding is that it is originally due to Lennart Meier (unpublished). 
Let $F$ denote the homotopy fiber of the usual $E_{\infty}$-ring spectrum map $tmf\rightarrow ko$. 
The rational stable homotopy groups of $F$ are then given by 
\[ \pi_*(F)\otimes_{\mathbb{Z}}\mathbb{Q} \cong S_{2*}\otimes_{\mathbb{Z}}\mathbb{Q} \oplus \Sigma^3 \mathbb{Q},\]
the graded $\mathbb{Q}$-vector space of level $1$ cusp forms\footnote{The cusp forms do not form a ring, since, for example, there is no weight $0$ cusp form, hence no multiplicative unit.}, with grading given by weight multiplied\footnote{Throughout, we will consistly write $S_{2*}$ to mean the graded $\mathbb{Z}[\frac{1}{6}]$-module of level $1$ cusp forms over $\mathbb{Z}\left[\frac{1}{6}\right]$ with grading given by weight multiplied by $2$. We write $S_{-2*}$ for the same graded $\mathbb{Z}[\frac{1}{6}]$-module but with the grading degrees also multiplied by $-1$. 
We write $S_k$ for the $\mathbb{Z}\left[\frac{1}{6}\right]$-module of weight $k$ level $1$ cusp forms over $\mathbb{Z}\left[\frac{1}{6}\right]$.} by $2$,
along with an additional copy of $\mathbb{Q}$ in degree $3$ coming from the Bott class in $\pi_4(ko)$ failing to be in the image of the map $\pi_*(tmf)\rightarrow\pi_*(ko)$, essentially because the weight $2$ Eisenstein series fails to be a modular form.

In fact the above is true after inverting the primes $2$ and $3$; it is not necessary to pass all the way to the rationalization. In particular, we have $\pi_n(F\left[\frac{1}{6}\right]) \cong 0$ if $n<3$, while $\pi_3(F\left[\frac{1}{6}\right])\cong \mathbb{Z}\left[\frac{1}{6}\right]$. We then have $\pi_n(F\left[\frac{1}{6}\right]) \cong 0$ for $n=4, 5, 6,7,8, \dots , 23$, and then we get $\pi_{24}(F\left[\frac{1}{6}\right])\cong \mathbb{Z}\left[\frac{1}{6}\right]$ generated by $\Delta$, and thereafter (i.e., for $n\geq 24$) we have that $\pi_n(F\left[\frac{1}{6}\right])$ is a free $\mathbb{Z}\left[\frac{1}{6}\right]$-module of rank $j$ if $n=24j+r$ with $r\equiv 0,8,12,16,20$ modulo $24$, rank $j-1$ if $n=24j+r$ with $r\equiv 2,4,6,10,14,18,22$ modulo $24$, and rank $0$ otherwise.
To narrow our focus to the rank of degrees where the homotopy groups of $F$ describe (away from $6$) exactly the classical level $1$ cusp forms, we define $\tcf$ as the $4$-connective cover of $F$. This yields the isomorphism of graded $\mathbb{Z}[\frac{1}{6}]$-modules 
$\pi_*(\tcf[\frac{1}{6}])\cong S_{2*}$. So, after inverting $6$, the homotopy groups of $\tcf$ are precisely the classical level $1$ cusp forms over $\mathbb{Z}[\frac{1}{6}]$.

It will not be important for the rest of this note, but it is perhaps satisfying to know that the spectrum $\tcf$ is also a $\tmf$-module, and the resulting action of $\pi_*(\tmf[\frac{1}{6}])$ on $\tcf[\frac{1}{6}]$ coincides with the usual action, by multiplication, of the ring $M_*$ of holomorphic level $1$ modular forms on its submodule of cusp forms. To see that this is so, first observe that the map $\tmf\rightarrow ko$ is a ring spectrum map, hence in particular a $\tmf$-module map, so its homotopy fiber $F$ is also a $\tmf$-module. 
Since $\tmf$ is a connective $E_{\infty}$-ring spectrum, we can take any $\tmf$-module $X$ and attach $\tmf$-cells (rather than $S^0$-cells, as in the stable version of the familiar ``attaching cells to kill higher homotopy'' construction from classical homotopy theory) to kill all the homotopy above some degree $n$, yielding a map of $\tmf$-modules $X\rightarrow X(-\infty,n]$. Our spectrum $\tcf$ is the fiber of the map $F\rightarrow F(-\infty,3]$, so $\tcf$ is the fiber of a map of $\tmf$-modules, so $\tcf$ is a $\tmf$-module. It is routine to check that the resulting $\pi_*(\tmf[\frac{1}{6}])$-action on $\pi_*(\tcf[\frac{1}{6}])$ is as expected.

\section{Hecke operators on rational topological cusp forms.}

When we speak of an ``action of Hecke operators'' on a spectrum $E$, we will always mean that the relevant Hecke operator $\tilde{T}_n$ ought to be a {\em homotopy class} of maps of spectra $E\rightarrow E$; we do not require Hecke operators to be defined ``on the nose,'' and we only require the usual Hecke relations $\tilde{T}_m\tilde{T}_n = \tilde{T}_n\tilde{T}_m$ and 
\begin{equation}\label{adams-hecke rel} \tilde{T}_{p^{r+2}}(x) = \tilde{T}_p\left(\tilde{T}_{p^{r+1}}(x)\right) - \frac{1}{p}\Psi^p\tilde{T}_{p^r}(x)\end{equation} to hold {\em up to homotopy}\footnote{The classical Hecke relation is $T_{p^{r+2}} = T_pT_{p^{r+1}} - p^{k-1}T_{p^r}$, which holds for the action of the Hecke operators on the ring of holomorphic modular forms. However, Baker points out in \cite{MR1037690} that his topological Hecke operators on elliptic homology instead satisfy the more general relation \eqref{adams-hecke rel}, where $\Psi^p$ here is the $p$th Adams operation. In the case of the holomorphic elliptic homology of the zero-sphere $S^0$, the action of $\Psi^p$ on $\ell_{2k}(S^0)\cong M_{k}$ is simply by multiplication by $p^k$, so \eqref{adams-hecke rel} reduces to the classical relation $T_{p^{r+2}} = T_pT_{p^{r+1}} - p^{k-1}T_{p^r}$ when the topological Hecke operators are evaluated on the zero-sphere. See \cite{top-hecke} for further discussion.}. Consequently, in the rest of this section, we will always be working in the stable homotopy category.

Getting an action of Hecke operators on $tcf$ is quite nontrivial: in \cite{MR1037690}, Baker constructs an action of Hecke operators on $TMF\left[\frac{1}{6}\right]$, and essentially the same method yields an action of Hecke operators on $tmf\left[\frac{1}{6}\right]$. Baker's method uses the complex-orientability of $TMF\left[\frac{1}{6}\right]$ in an essential way, and it does not yield Hecke operations on $TMF$ (or on $Tmf$ or $tmf$) without inverting the primes $2$ and $3$, since it is only after inverting $6$ that these spectra become complex-oriented.

Like $TMF$ and $Tmf$ and $tmf$, the spectrum $tcf$ is not complex-oriented; but unlike $TMF$ and $Tmf$ and $tmf$, inverting $6$ still does not make $tcf$ complex-oriented\footnote{This is simply because a spectrum cannot be complex-oriented without having a ring structure. Since $\pi_0(tcf)$ remains trivial after inverting any set of primes, it is not possible to make $tcf$ into a ring spectrum by inverting any set of primes.}. Some other idea is needed to build Hecke operators on $tcf$, even on $tcf\left[\frac{1}{6}\right]$.

In this note we take the path of least resistance: we simply rationalize, and then try to show that Baker's Hecke operators on rational $tmf$ restrict to operators on rational $tcf$. This is, however, straightforward: for each positive integer $n$, we have the solid arrows in the diagram of spectra whose rows are homotopy fiber sequences 
\begin{equation}\label{diag 095095}\xymatrix{
 F\smash H\mathbb{Q} \ar[r]\ar@{-->}[d] & tmf \smash H\mathbb{Q} \ar[r]\ar[d]_{\tilde{T}_n}  & ko \smash H\mathbb{Q} \\
 F\smash H\mathbb{Q} \ar[r] & tmf \smash H\mathbb{Q} \ar[r] & ko \smash H\mathbb{Q}
}\end{equation}
and the composite $F\smash H\mathbb{Q}\rightarrow ko\smash H\mathbb{Q}$ from the upper left corner to the lower right corner induces the zero map in $\pi_*$, since Baker's $\tilde{T}_n$ agrees with the classical Hecke operator on $T_n$ on $\pi_*(tmf\left[\frac{1}{6}\right])$. In rational spectra, a map that induces zero in $\pi_*$ is nulhomotopic, so there exists a map filling in the dotted line in diagram \eqref{diag 095095} and making it commute. The number of such maps is the order of the cokernel of the map \[[ F\smash H\mathbb{Q},\Sigma^{-1} tmf\smash H\mathbb{Q}] \rightarrow [F\smash H\mathbb{Q},\Sigma^{-1} ko\smash H\mathbb{Q}].\]
Since the rational homotopy groups of $\Sigma^{-1} ko$ are concentrated in odd degrees and since the only odd-degree rational homotopy group of $F$ is the copy of $\mathbb{Q}$ in rational $\pi_3(F)$ coming from the failure of $\pi_4(tmf)\rightarrow \pi_4(ko)$ to be rationally surjective, the dotted arrow in \eqref{diag 095095} is well-defined up to the question of how $T_n$ ought to act on $\pi_3(F)\otimes_{\mathbb{Z}}\mathbb{Q}$; a reasonable way to make this choice might be to use the action of $T_n$ on the quasimodular form $E_2$, as in \cite{MR3373249}. For the sake of this note, one can make whatever choice one wants, since we make no use of the Hecke action on $\pi_3(F)$; since taking the $4$-connective cover is a functor, each map $F\smash H\mathbb{Q}\rightarrow F\smash H\mathbb{Q}$ yields a map $\tcf \smash H\mathbb{Q}\rightarrow \tcf\smash H\mathbb{Q}$ with the same effect on homotopy in degrees $\geq 4$, yielding our Hecke action on $\tcf \smash H\mathbb{Q}$.

Consequently, for each positive integer $n$, we have a topological Hecke operator $\tilde{T}_n$ on $tcf$. It is a map of spectra $\tilde{T}_n: tcf\smash H\mathbb{Q} \rightarrow tcf\smash H\mathbb{Q}$, well-defined up to homotopy. By applying the base-change functor $-\smash_{H\mathbb{Q}} H\mathbb{C}$ we get a map of spectra $tcf\smash H\mathbb{C} \rightarrow tcf\smash H\mathbb{C}$, which by abuse of notation we will also call $\tilde{T}_n$, and which plays an important role in the next section.

\section{The Petersson product on complexified topological cusp forms.}

Given a connective spectrum $X$ and an integer $n$, let $X[0,n]$ denote $X$ with cells attached to kill all homotopy groups in degrees $>n$.
For each integer $k$, let $\tcf^{\wt k}_{\mathbb{C}}$ be the fiber of the map from $\left( \tcf\smash H\mathbb{C}\right)[0,2k]$ to $\left( \tcf\smash H\mathbb{C}\right)[0,2k-1]$. Since $\tcf\smash H\mathbb{C}$ is a rational spectrum, the fiber inclusion map $\tcf^{\wt k}_{\mathbb{C}}\rightarrow \left( \tcf\smash H\mathbb{C}\right)[0,2k]$ splits, and a degree argument shows that the splitting is unique. Hence
$\tcf^{\wt k}_{\mathbb{C}}$ is uniquely a wedge summand of $\tcf\smash H\mathbb{C}$, homotopy equivalent to the $2k$th suspension of the Eilenberg-Mac Lane spectrum of the $\mathbb{C}$-vector space $S_k\otimes_{\mathbb{Z}}\mathbb{C}$ of weight $k$ level $1$ classical cusp forms.

For each integer $k$, let 
\[\petersson_k: \Sigma^{-2k}\tcf^{\wt k}_{\mathbb{C}} \smash_{H\mathbb{R}} \Sigma^{-2k}\tcf^{\wt k}_{\mathbb{C}} \rightarrow H\mathbb{C} \] be the Eilenberg-Mac Lane spectrum functor $H$ applied to the Petersson product map $\left(S_k\otimes_{\mathbb{Z}}\mathbb{C}\right) \otimes_{\mathbb{R}} \left(S_k\otimes_{\mathbb{Z}}\mathbb{C}\right) \rightarrow \mathbb{C}$.

Now we are ready to define the Petersson product on topological cusp forms:
\begin{definition}\label{def of topological petersson bichar}
For each integer $j$, the {\em weight $j$ topological Petersson bicharacter} is the composite map of spectra 
\begin{align*} 
 \Sigma^{-2j}\tcf \smash \Sigma^{-2j}\tcf 
  &\rightarrow \left( \Sigma^{-2j}\tcf\smash H\mathbb{C}\right) \smash_{H\mathbb{R}} \left( \Sigma^{-2j}\tcf\smash H\mathbb{C}\right) \\
  &\rightarrow \Sigma^{-2j}\tcf^{\wt j}_{\mathbb{C}} \smash_{H\mathbb{R}} \Sigma^{-2j}\tcf^{\wt j}_{\mathbb{C}} \\
  &\stackrel{\petersson_j}{\longrightarrow} H\mathbb{C}.
\end{align*}

 Consequently, given an unpointed topological space\footnote{We cannot make the same construction for an arbitrary spectrum $X$, since we are required to have a diagonal map on $X$.} $X$, an integer or half-integer $k$, and a pair of elements $f,g\in [\Sigma^{\infty}X_+,\Sigma^{-2k}\tcf]\cong \tcf^{-2k}(X)$, the composite map
\begin{align*}
 \Sigma^{\infty} X_+ 
  &\stackrel{\Sigma^{\infty}\Delta_+}{\longrightarrow}\Sigma^{\infty} (X\times X)_+ \\
  &\stackrel{\cong}{\longrightarrow} \Sigma^{\infty}X_+\smash \Sigma^{\infty}X_+ \\
  &\stackrel{f\smash g}{\longrightarrow} \Sigma^{-2k}\tcf\smash \Sigma^{-2k} \tcf \\
  &\stackrel{}{\longrightarrow} \Sigma^{-2k}\tcf\smash_{H\mathbb{R}} \Sigma^{-2k} \tcf \\
  &\stackrel{\Sigma^{4(j-k)}\petersson_j}{\longrightarrow} \Sigma^{2(j-k)}H\mathbb{C}
\end{align*}
is an element of $\tilde{H}^{4(j-k)}(X_+; \mathbb{C})\cong H^{4(j-k)}(X; \mathbb{C})$. We define $\langle f,g\rangle_j$,
the {\em weight $j$ topological Petersson product of $f$ and $g$}, to be  
that element of $H^{4(j-k)}(X;\mathbb{C})$.
\end{definition}
Note that the topological Petersson products are {\em not} stable invariants. In particular, suspending $X$ tends to cause topological Petersson products not supported on the basepoint to become zero.

%

%

\begin{examples}\label{product examples}\leavevmode
\begin{itemize}
\item
When $X = \pt$, an element $f$ of $\left[\Sigma^{\infty}X_+,\Sigma^{-2k}\tcf\left[\frac{1}{6}\right]\right]\cong \pi_{2k}\left(\tcf\left[\frac{1}{6}\right]\right)$ is given by a weight $k$ level $1$ classical cusp form.
Since $\pi_0(\Sigma^{-2k}\tcf^{\wt j}_{\mathbb{C}})$ vanishes unless $j=k$,
the composite 
\begin{align*}
 \Sigma^{\infty} X_+
  &\stackrel{\cong}{\longrightarrow} \Sigma^{\infty} S^0 \\
  &\stackrel{\cong}{\longrightarrow} \Sigma^{\infty} S^0 \smash \Sigma^{\infty} S^0  \\
  &\stackrel{f\smash g}{\longrightarrow} \Sigma^{-2k}\tcf\smash \Sigma^{-2k} \tcf \\
  &\rightarrow \Sigma^{-2k}\tcf^{\wt j}_{\mathbb{C}} \smash_{H\mathbb{R}} \Sigma^{-2k}\tcf^{\wt j}_{\mathbb{C}} 
\end{align*}
is nulhomotopic unless $j=k$.
Consequently $\langle f,g\rangle_j =0$ unless $j=k$. When $j=k$, the topological Petersson product is nonzero:
it is immediate from Definition \ref{def of topological petersson bichar} that $\langle f,g\rangle_k = \langle f,g\rangle \in \mathbb{C}\cong H^0(\pt;\mathbb{C})$, the classical Petersson product of the cusp forms $f$ and $g$.
\item More generally, when $X$ is a finite discrete space with $n$ points, then 
$\tcf^{2k}(X)$ splits as a direct sum of $n$ copies of $\tcf^{2k}(\pt)$, and the weight $k$ topological Petersson product on each summand of $\tcf^{2k}(X)\left[\frac{1}{6}\right]$ is simply the classical Petersson product, while the weight $j$ topological Petersson product is zero for $j\neq k$.
\item
When $X=S^{n}$ for $n>0$, it is standard that the composite map of spectra
\begin{align*}
 \Sigma^{\infty} S^{n} \vee \Sigma^{\infty} S^{0} 
  &\stackrel{\cong}{\longrightarrow} \Sigma^{\infty} S^{n}_+ \\
  &\stackrel{\Sigma^{\infty}\Delta_+}{\longrightarrow}\Sigma^{\infty} (S^{n}\times S^{n})_+ \\
  &\stackrel{\simeq}{\longrightarrow} \Sigma^{\infty}S^{2n} \vee \Sigma^{\infty}S^{n} \vee \Sigma^{\infty}S^{n} \vee \Sigma^{\infty}S^{0}   
\end{align*}
sends the summand $\Sigma^{\infty} S^{n}$ in its domain via the categorical diagonal map to the summands $\Sigma^{\infty}S^{n} \vee \Sigma^{\infty}S^{n}$ in the codomain, and sends the summand $\Sigma^{\infty}S^0$ in its domain via the identity map to the summand $\Sigma^{\infty}S^0$ in its codomain.
Consequently, given elements $f,g\in \left[\Sigma^{\infty}S^{n}_+,\Sigma^{-2k}\tcf\left[\frac{1}{6}\right]\right]$, 
we decompose $f$ into a weight $k$ cusp form $f_0 \in \left[\Sigma^{\infty}S^0,\Sigma^{-2k}\tcf\left[\frac{1}{6}\right]\right]$ and a weight $k+\frac{n}{2}$ cusp form $f_{n/2} \in \left[\Sigma^{\infty}S^{n},\Sigma^{-2k}\tcf\left[\frac{1}{6}\right]\right]$, and similarly we decompose $g$ into a weight $k$ cusp form $g_0\in \left[\Sigma^{\infty}S^0,\Sigma^{-2k}\tcf\left[\frac{1}{6}\right]\right]$ and a weight $k+\frac{n}{2}$ cusp form $g_{n/2} \in \left[\Sigma^{\infty}S^{n},\Sigma^{-2k}\tcf\left[\frac{1}{6}\right]\right]$, so that 
the composite map
\begin{align*}
 \Sigma^{\infty} S^{n}_+ 
  &\longrightarrow \Sigma^{\infty} S^{n}_+ \smash \Sigma^{\infty} S^{n}_+ \\
  &\stackrel{f\smash g}{\longrightarrow} \Sigma^{-2k}\tcf\left[\frac{1}{6}\right]\smash \Sigma^{-2k} \tcf\left[\frac{1}{6}\right] \\
  &\rightarrow \Sigma^{-2k}\tcf^{\wt j}_{\mathbb{C}} \smash_{H\mathbb{R}} \Sigma^{-2k}\tcf^{\wt j}_{\mathbb{C}} \\
  &\rightarrow \Sigma^{4(j-k)}H\mathbb{C} 
\end{align*}
is nulhomotopic unless $j=k$.
To be clear: the composite map
\begin{align}
\label{comp map 1} \Sigma^{\infty} S^{2n}\vee 
 \Sigma^{\infty} S^{n}\vee 
 \Sigma^{\infty} S^{n}\vee 
 \Sigma^{\infty} S^{0} 
 &\stackrel{\simeq}{\longrightarrow} \Sigma^{\infty} (S^n\times S^n)_+ \\
\label{comp map 2} &\stackrel{\simeq}{\longrightarrow} \Sigma^{\infty} S^n_+\smash \Sigma^{\infty} S^n_+ \\
\label{comp map 3} &\stackrel{f\smash g}{\longrightarrow} \Sigma^{-2k}tcf^{\wt j}_{\mathbb{C}} \smash_{H\mathbb{R}}\Sigma^{-2k} tcf^{\wt j}_{\mathbb{C}}\\
 & \rightarrow \Sigma^{4(j-k)}H\mathbb{C}\end{align}
has the following effect on each wedge summands in its domain:
\begin{itemize}
\item On $\Sigma^{\infty} S^{2n}$, the composite map \[\Sigma^{\infty} S^{2n} \rightarrow \Sigma^{-2k}tcf^{\wt j}_{\mathbb{C}} \smash_{H\mathbb{R}}\Sigma^{-2k} tcf^{\wt j}_{\mathbb{C}}\rightarrow \Sigma^{4(j-k)}H\mathbb{C}\] is $\langle f_{n/2},g_{n/2}\rangle_{j}$, a classical Petersson product of weight $k+n/2$ cusp forms, hence certainly capable of being nonzero if (and only if) $j = k+ n/2$.
\item On one summand $\Sigma^{\infty}S^n$, the composite map \[\Sigma^{\infty} S^{n} \rightarrow \Sigma^{-2k}tcf^{\wt j}_{\mathbb{C}} \smash_{H\mathbb{R}}\Sigma^{-2k} tcf^{\wt j}_{\mathbb{C}}\rightarrow \Sigma^{4(j-k)}H\mathbb{C}\] is $\langle f_{0},g_{n/2}\rangle_{j}$, a classical Petersson product of cusp forms of distinct weights, hence is zero.
\item On the other summand $\Sigma^{\infty}S^n$, the composite map \[\Sigma^{\infty} S^{n} \rightarrow \Sigma^{-2k}tcf^{\wt j}_{\mathbb{C}} \smash_{H\mathbb{R}}\Sigma^{-2k} tcf^{\wt j}_{\mathbb{C}}\rightarrow \Sigma^{4(j-k)}H\mathbb{C}\] is $\langle f_{n/2},g_{0}\rangle_{j}$, again a classical Petersson product of cusp forms of distinct weights, hence again zero.
\item On the summand $\Sigma^{\infty}S^0$, the composite map \[\Sigma^{\infty} S^{0} \rightarrow \Sigma^{-2k}tcf^{\wt j}_{\mathbb{C}} \smash_{H\mathbb{R}}\Sigma^{-2k} tcf^{\wt j}_{\mathbb{C}}\rightarrow \Sigma^{4(j-k)}H\mathbb{C}\] is $\langle f_{0},g_{0}\rangle_{j}$, a classical Petersson product of weight $k$ cusp forms, hence capable of being nonzero if (and only if) $j = k$. 
\end{itemize}
However, the topological Petersson product of $f$ and $g$ is the composite of \eqref{comp map 2} and \eqref{comp map 3}, {\em precomposed with the diagonal map} $\Delta: \Sigma^{\infty} S^{n}_+\rightarrow \Sigma^{\infty} (S^n\times S^n)_+$. While the composite of \eqref{comp map 1} through \eqref{comp map 3} is non-nulhomotopic on the top-dimensional summand $\Sigma^{\infty}S^{2n}$, precomposing with $\Delta$ misses that top-dimensional summand entirely. (Of course this follows immediately from the triviality of $\pi_n(S^{2n})$. Recall that $n>0$ throughout this example.)

Consequently, on the $n$-sphere for $n>2$, the weight $j$ topological Petersson product of $f$ and $g$ is trivial unless $j=k$. If $j=k$, then the topological Petersson product of $f$ and $g$ is supported entirely on the basepoint of $S^n$, where it agrees with the classical Petersson product of $f_0$ and $g_0$. On the $n$-sphere for $n>0$, the topological Petersson product is completely insensitive to $f_{n/2}$ and $g_{n/2}$. In particular, the topological Petersson product is degenerate on $\tcf^*(S^{n})\otimes_{\mathbb{Z}}\mathbb{C}$.
\item The example of the sphere $S^{n}$ for $n>0$ may give the reader the impression that the topological Petersson product on $\tcf^*(X)\otimes_{\mathbb{Z}}\mathbb{C}$ is simply given by the classical Petersson product on the cusp forms supported on the basepoint of $X_+$.
This impression is not correct, as the following example demonstrates: let $X = \mathbb{C}P^2$. The composite map
\begin{align*} 
 \Sigma^{\infty} S^0 \vee \Sigma^{\infty} \mathbb{C}P^2 
  &\stackrel{\cong}{\longrightarrow} \Sigma^{\infty} \mathbb{C}P^2_+  \\
  &\rightarrow \Sigma^{\infty}(\mathbb{C}P^2 \times \mathbb{C}P^2)_+ \\
  &\stackrel{\cong}{\longrightarrow} \Sigma^{\infty}\mathbb{C}P^2_+\smash \Sigma^{\infty}\mathbb{C}P^2_+ \\
  &\stackrel{\cong}{\longrightarrow} \left(\Sigma^{\infty}\mathbb{C}P^2 \smash\Sigma^{\infty}\mathbb{C}P^2\right) \vee \Sigma^{\infty}\mathbb{C}P^2 \vee\Sigma^{\infty}\mathbb{C}P^2 \vee \Sigma^{\infty}S^0
\end{align*}
is 
non-nulhomotopic: it is nulhomotopic on the $2$-cell $S^2 \simeq \mathbb{C}P^1\subseteq \mathbb{C}P^2$, but the top cell of $\mathbb{C}P^2$ is mapped by a rational equivalence to the $2$-skeleton of $\mathbb{C}P^2\smash \mathbb{C}P^2$, since the cup square of a generator of $H^2(\mathbb{C}P^2;\mathbb{C})$ is a generator of $H^4(\mathbb{C}P^2;\mathbb{C})$.

Consequently, given elements $f,g\in \left[\Sigma^{\infty}\mathbb{C}P^2_+,\Sigma^{-2k}\tcf\left[\frac{1}{6}\right]\right]]$, we decompose $f$ into a weight $k$ cusp form $f_0 \in \left[\Sigma^{\infty}S^0,\Sigma^{-2k}\tcf\left[\frac{1}{6}\right]\right]$ and an element $f^{\prime}$ of $\left[\Sigma^{\infty}\mathbb{C}P^2,\Sigma^{-2k}\tcf\left[\frac{1}{6}\right]\right]$, and similarly we decompose $g$ into a weight $k$ cusp form $g_0 \in \left[\Sigma^{\infty}S^0,\Sigma^{-2k}\tcf\left[\frac{1}{6}\right]\right]$ and an element $g^{\prime}$ of $\left[\Sigma^{\infty}\mathbb{C}P^2,\Sigma^{-2k}\tcf\left[\frac{1}{6}\right]\right]$.
For the same reasons as in the previous examples, $\langle f_0,g_0\rangle_k = \langle f,g\rangle_k$, and $\langle f_0,g_0\rangle_j = 0$ if $j\neq k$. 

The Atiyah-Hirzebruch spectral sequence
\begin{align*}
 E_2^{s,t} \cong H^s\left(\mathbb{C}P^2; \tcf^t \otimes_{\mathbb{Z}}\mathbb{C}\right) 
  &\Rightarrow \tcf^{s+t}(\mathbb{C}P^2)\otimes_{\mathbb{Z}}\mathbb{C} \\
 d_r: E_r^{s,t} &\rightarrow E_r^{s+r,t-r+1} \end{align*}
collapses at $E_2$; one way to see this is by Dold's theorem that the shortest nonzero Atiyah-Hirzebruch differential is torsion-valued, from \cite{MR0198464} (see \cite{MR1193149} for a more easily-obtained English-language reference). 
This yields a splitting of $\mathbb{C}$-vector spaces
\begin{align*} 
 \tcf^{-2k}(\mathbb{C}P^2)\otimes_{\mathbb{Q}}\mathbb{C} 
  &\cong \left( H^0(\mathbb{C}P^2;\mathbb{C})\otimes_{\mathbb{Q}} S_k\right) \\ 
  &\oplus\left( H^2(\mathbb{C}P^2;\mathbb{C})\otimes_{\mathbb{Q}} S_{k+1}\right) \\
  &\oplus \left( H^4(\mathbb{C}P^2;\mathbb{C})\otimes_{\mathbb{Q}} S_{k+2}\right).\end{align*}
Consequently we can decompose the image of $f$ in $\tcf^{-2k}(\mathbb{C}P^2)\otimes_{\mathbb{Z}}\mathbb{C}$ into a triple $(f_0,f_2,f_4)$, where $f$ is a weight $k$ cusp form and $f_2$ is a weight $k+1$ cusp form and $f_4$ is a weight $k+2$ cusp form. This refines the above decomposition of $f$ into $f_0$ and $f^{\prime} = (f_2,f_4)$. We decompose $g$ into $(g_0,g_2,g_4)$ similarly.


The composite 
\begin{align*}
 \Sigma^{\infty} S^2
  &\stackrel{\cong}{\longrightarrow} \Sigma^{\infty} \mathbb{C}P^1\\
  &\stackrel{}{\longrightarrow} \Sigma^{\infty} \mathbb{C}P^2_+\\
  &\stackrel{}{\longrightarrow} \Sigma^{\infty} \mathbb{C}P^2_+ \smash \Sigma^{\infty} \mathbb{C}P^2_+  \\
  &\stackrel{f_2\smash g_2}{\longrightarrow} \Sigma^{-2k}\tcf\smash \Sigma^{-2k} \tcf \\
  &\rightarrow \Sigma^{-2k}\tcf^{\wt j}_{\mathbb{C}} \smash \Sigma^{-2k}\tcf^{\wt j}_{\mathbb{C}} \end{align*}
is nulhomotopic for all $j$, so the component of $\langle f,g\rangle_j$ on the $2$-cell in $\mathbb{C}P^2_+$ is trivial for all $j$.

The $4$-cell in $\mathbb{C}P^2_+$ is the interesting part: by the above cell-by-cell analysis of the homotopy class of the diagonal map on $\mathbb{C}P^2_+$, the component of $\langle f,g\rangle_j$ on the $4$-cell in $\mathbb{C}P^2_+$ is $\langle f_2,g_2\rangle_{j+1}$.
Finally, we have 
\begin{align*}
  \langle f,g\rangle_k &= \langle f_0,g_0\rangle \in \mathbb{C}\cong H^0(\mathbb{C}P^2; \mathbb{C}) \\
  \langle f,g\rangle_{k+1} &= \langle f_2,g_2\rangle \in \mathbb{C}\cong H^4(\mathbb{C}P^2; \mathbb{C}) \\
  \langle f,g\rangle_{j} &= 0 \mbox{\ if\ } j\neq k,k+1.
\end{align*}
\end{itemize}
\end{examples}

The above examples show that, since $\tcf^{\wt j}_{\mathbb{C}}$ is a rational spectrum, the topological Petersson product on $\tcf^*(X)\left[\frac{1}{6}\right]$ is determined on the $n$-skeleton $\tcf^*(X^n)\left[\frac{1}{6}\right]$ by the images of the {\em rational} cells in $X^n\smash H\mathbb{Q}$ under the rationalized diagonal map $X \rightarrow X \times X$, up to elements of higher Atiyah-Hirzebruch filtration\footnote{Recall that the {\em Atiyah-Hirzebruch filtration on $E^*(X)$,} for a generalized cohomology theory $E^*$ and a CW-complex $X$, is the decreasing filtration of $E^*(X)$ whose $n$th layer $F^n$ is the set of elements of $E^*(X)$ whose image under the map $E^*(X) \rightarrow E^*(X/X^{n-1})$ is zero. Here $X/X^{n-1}$ is $X$ with its $(n-1)$-skeleton $X^{n-1}$ collapsed (``pinched'') down to a point. An element of $E^*(X)$ has {\em Atiyah-Hirzebruch degree $n$} if it has Atiyah-Hirzebruch filtration $n$ but not Atiyah-Hirzebruch filtration $n-1$.}; that is, on the associated graded of the Atiyah-Hirzebruch filtration on $\tcf^*(X)\otimes_{\mathbb{Z}}\mathbb{C}$, the topological Petersson product is determined by the cup product on $H^*(X;\mathbb{C})$ and the classical Petersson product.
 Consequently, we have:
\begin{prop}\label{main prop}
Let $X$ be an unpointed finite CW-complex. 
Let $k$ be an integer or half-integer, let $f,g\in \tcf^{-2k}(X)\left[\frac{1}{6}\right]$, and let $j$ be an integer.
Then $\langle f,g\rangle_j$ is the element of $H^{4(j-k)}(X;\mathbb{C})\cong \hom_{\mathbb{Q}}(H_{4(j-k)}(X;\mathbb{Q}),\mathbb{C})$ given, up to higher Atiyah-Hirzebruch filtration, by composing the rational homology coproduct map
\begin{align*} H_{4(j-k)}(X;\mathbb{Q}) &\rightarrow \coprod_n H_{n}(X;\mathbb{Q})\otimes_{\mathbb{Q}} H_{4(j-k)-n}(X;\mathbb{Q})\end{align*}
with the summand projection 
\begin{align*} \coprod_n H_{n}(X;\mathbb{Q})\otimes_{\mathbb{Q}} H_{4(j-k)-n}(X;\mathbb{Q}) &\rightarrow H_{2(j-k)}(X; \mathbb{Q})\otimes_{\mathbb{Q}} H_{2(j-k)}(X; \mathbb{Q})
\end{align*}
followed by the map 
\begin{align}
\label{map 340943} H_{2(j-k)}(X; \mathbb{Q})\otimes_{\mathbb{Q}} H_{2(j-k)}(X; \mathbb{Q}) &\rightarrow \left( S_{j} \otimes_{\mathbb{Z}} \mathbb{Q} \right)\otimes_{\mathbb{Q}} \left( S_{j} \otimes_{\mathbb{Z}} \mathbb{Q} \right)
\end{align}
given by the chain of isomorphisms
\begin{align*}
 \tcf^{-2k}(X)\otimes_{\mathbb{Z}}\mathbb{C} 
  \cong \coprod_n H^{n}(X; \tcf^{-n-2k}\otimes_{\mathbb{Z}}\mathbb{C}) \\
  \cong \coprod_n \hom_{\mathbb{Q}}(H_{n}(X;\mathbb{Q}), \tcf^{-n-2k} \otimes_{\mathbb{Z}} \mathbb{C}) \\
  \cong \coprod_n \hom_{\mathbb{Q}}(H_{n}(X;\mathbb{Q}), S_{\frac{n}{2}+k} \otimes_{\mathbb{Z}} \mathbb{C}),
\end{align*}
and finally postcomposing \eqref{map 340943} with the classical Petersson product map 
\begin{align*}
\left( S_{j} \otimes_{\mathbb{Z}} \mathbb{Q} \right)\otimes_{\mathbb{Q}} \left( S_{j} \otimes_{\mathbb{Z}} \mathbb{Q} \right)
 &\rightarrow \mathbb{C}.
\end{align*}
\end{prop}
We can rephrase Proposition \ref{main prop} as the following method for calculating $\langle -,-\rangle_j$, not on $\tcf^*(X)\otimes_{\mathbb{Z}}\mathbb{C}$, but only on {\em the associated graded of the Atiyah-Hirzebruch filtration on} $\tcf^*(X)\otimes_{\mathbb{Z}}\mathbb{C}$. Given elements $f,g\in \tcf^{-2k}(X)\otimes_{\mathbb{Z}}\mathbb{C}$, we use the isomorphism $\tcf^*(X)\otimes_{\mathbb{Z}}\mathbb{C}\cong \tcf^*(S^0)\otimes_{\mathbb{Z}}H^*(X; \mathbb{C})$ to express $f$ as a sum $\sum_h f_h \otimes x_h$ and to express $g$ as a sum $\sum_i g_i\otimes y_i$, with each $f_h$ and each $g_i$ a classical cusp form, and with each $x_h$ and $y_i$ an element of $H^*(X;\mathbb{C})$. Since we are only calculating $\langle -,-\rangle_j$ on the associated graded of the Atiyah-Hirzebruch filtration, we can assume without loss of generality that $f$ and $g$ are each homogeneous with respect to Atiyah-Hirzebruch filtration, i.e., that the integer $\left| x_h\right|$ is independent of $h$, and the integer $\left| y_i\right|$ is independent of $i$.
Then Proposition \ref{main prop} yields the formula
\begin{align}
\label{product formula 200}
 \langle f,g\rangle_j &\equiv \sum_{(h,i): \left|x_h\right| = \left|y_i\right| = 2(j-k)} \langle f_h,g_i\rangle \cdot (x_h\cup y_i)
\end{align}
modulo terms of greater Atiyah-Hirzebruch filtration. 
The form $\langle -,-\rangle$ on the right-hand side of \eqref{product formula 200} is the classical Petersson product of cusp forms, while $x_h\cup y_i$ is the usual cup product in $H^*(X;\mathbb{C})$. 

\begin{corollary}\label{degeneracy of the petersson product}
Let $X$ be a finite CW-complex such that $H^m(X;\mathbb{Q})$ is nonzero for some positive integer $m$.
Then there exists some integer $n$ such that, for every integer $j$, the topological Petersson product of weight $j$ is degenerate on $\tcf^n(X)\otimes_{\mathbb{Z}}\mathbb{C}$.
\end{corollary}
\begin{proof}
Let $d$ be the greatest integer $m$ such that $H^m(X;\mathbb{Q})$ is nonzero. Then $d>0$, so for degree reasons, there exists no element $y\in H^d(X;\mathbb{Q})$ such that $x\cup y \neq 0$. Let $f$ be any nonzero classical cusp form, of any weight $k$. Then 
\[ 0\neq f\otimes x \in \tcf^{-2k}(S^0) \otimes_{\mathbb{Z}} H^d(X;\mathbb{Q})\otimes_{\mathbb{Z}}\mathbb{C} \cong 
\tcf^{d-2k}(X)\otimes_{\mathbb{Z}}\mathbb{C},\]
and by formula \eqref{product formula 200} and the triviality of $x\cup y$ for all $y\in H^d(X;\mathbb{Q})$, we have that $\langle f,g\rangle_j$ is zero modulo terms of greater Atiyah-Hirzebruch filtration, for any $g$. But since $d$ is the highest degree in which $H^*(X;\mathbb{C})$ vanishes, $f$ is already in the highest possible Atiyah-Hirzebruch filtration. So $\langle f,g\rangle_j$ is zero for all $g$. So $\langle -,-\rangle_j$ is degenerate on $\tcf^{d-2k}(X)\otimes_{\mathbb{Z}}\mathbb{C}$.
\end{proof}

A consequence of formula \eqref{product formula 200} is that the weight $j$ topological Petersson product $\langle -,-\rangle_j$ vanishes on the weight $k$ cusp forms supported on cells of dimension $>2(j-k)$. In light of that consequence, we make the following definition of the topological Petersson product {\em simpliciter}, rather than the topological Petersson product of a particular weight:
\begin{definition}\label{def of topological petersson product}
Let $X$ be a finite CW-complex.
By the {\em topological Petersson product} on $\tcf^*(X)\otimes_{\mathbb{Z}}\mathbb{C}$ we mean the $\mathbb{R}$-bilinear form 
\begin{align*}\left( \tcf^*(X)\otimes_{\mathbb{Z}}\mathbb{C} \right) \otimes_{\mathbb{R}}\left( \tcf^*(X)\otimes_{\mathbb{Z}}\mathbb{C} \right) &\rightarrow \mathbb{C}\end{align*}
given on the elements of $\tcf^{i}(X)\otimes_{\mathbb{Z}}\mathbb{C}$ of Atiyah-Hirzebruch degree $j$ by the weight $(j-i)/2$ topological Petersson product $\langle -,-\rangle_{(j-i)/2}$. (If $j-i$ is odd, then this product is defined to be zero.) We write $\langle -,-\rangle_{X}$ for the topological Petersson product on the space $X$.
\end{definition}
\begin{example}
In the case of the complex projective plane, collapse of the Atiyah-Hirzebruch spectral sequence gives us that the complexified $tcf$-cohomology $\tcf^{-*}(\mathbb{C}P^2)\otimes_{\mathbb{Z}}\mathbb{C}$ is isomorphic to the direct sum of:
\begin{itemize}
\item a copy of $S_*\otimes_{\mathbb{Z}}\mathbb{C}$ supported on the $0$-cell of $\mathbb{C}P^2$,
\item a copy of $S_{2+*}\otimes_{\mathbb{Z}}\mathbb{C}$ supported on the $2$-cell of $\mathbb{C}P^2$, 
\item and a copy of $S_{4+*}\otimes_{\mathbb{Z}}\mathbb{C}$ supported on the $4$-cell of $\mathbb{C}P^2$.
\end{itemize}
As in Examples \ref{product examples}, an element $f$ of $\tcf^{-2k}(\mathbb{C}P^2)\otimes_{\mathbb{Z}}\mathbb{C}$ is given by a triple $(f_k,f_{k+1},f_{k+2})$, where $f_k,f_{k+1},f_{k+2}$ are weight $k,k+1$, and $k+2$ classical cusp forms, respectively. These have Atiyah-Hirzebruch filtrations $0,2,$ and $4$ respectively: this is simply a matter of reading off the dimensions of the cells on which $f_k,f_{k+1},$ and $f_{k+2}$ are supported. 
Consequently it follows from Examples \ref{product examples} and Definition \ref{def of topological petersson product} that the topological Petersson product $\langle f,g\rangle_{\mathbb{C}P^2}$ is the nonhomogeneous element $\langle f_k,g_k\rangle + \langle f_{k+1},g_{k+1}\rangle\in H^*(\mathbb{C}P^2;\mathbb{C})$. Here $\langle f_k,g_k\rangle$ is the classical Petersson product of $f_k$ and $g_k$, regarded as an element of $\mathbb{C}\cong H^0(\mathbb{C}P^2;\mathbb{C})$, while $\langle f_{k+1},g_{k+1}\rangle$ is the classical Petersson product of $f_{k+1}$ and $g_{k+1}$, regarded as an element of $\mathbb{C}\cong H^4(\mathbb{C}P^2;\mathbb{C})$.

Note that the topological Petersson product over $\mathbb{C}P^2$ is degenerate, since it is insensitive to the part of the topological cusp form supported on the top cell in $\mathbb{C}P^2$, a phenomenon we noticed already in Examples \ref{product examples}. 
\end{example}

\begin{theorem}\label{compact kahler mfld thm}
Let $X$ be a compact K\"{a}hler manifold of complex dimension $d$. Write $\tcf^*_{\mathbb{C}}(X)^{\leq 2d/3}$ for the sub-$\mathbb{C}$-vector space of $\tcf^*(X)\otimes_{\mathbb{Z}}\mathbb{C}$ spanned by the elements of Atiyah-Hirzebruch degree $n$ such that $n$ is even and $n\leq 2d/3$. Then the topological Petersson product is a nondegenerate $\mathbb{R}$-bilinear form on $\tcf^*_{\mathbb{C}}(X)^{\leq 2d/3}$.
\end{theorem}
\begin{proof}
Let $\omega$ be a K\"{a}hler $2$-form on $X$. 
Given an integer $n$ such that $n < d$, the hard Lefschetz theorem gives us that multiplication by $[\omega]^{d-n}$ yields an isomorphism $H^n(X;\mathbb{C})\rightarrow H^{2d-n}(X;\mathbb{C})$.
Since we assume $n\leq 2d/3$, we have $2n\leq 2d-n$, and consequently multiplication by $[\omega]^{d-n}$ factors through the multiplication by $[\omega]^{n/2}$ map
\begin{align}\label{mult by power of omega} H^n(X;\mathbb{C})&\rightarrow H^{2n}(X;\mathbb{C}).\end{align}.
Since an isomorphism factors through \eqref{mult by power of omega}, we have that \eqref{mult by power of omega} is injective. Now given a nonzero element \[\sum_h f_h \otimes x_h\in \tcf^*(S^0)\otimes_{\mathbb{Z}}H^*(X; \mathbb{C})\cong \tcf^*(X)\otimes_{\mathbb{Z}}\mathbb{C}\]
with the elements $\{ x_h\}\subseteq H^*(X;\mathbb{C})$ $\mathbb{C}$-linearly independent, choose, for each $h$, a classical cusp form $g_h\in \tcf^*(S^0)$ such that the classical Petersson product $\langle f_h,g_h\rangle$ is nonzero. This is possible since the classical Petersson product is nondegenerate. Let $g = \sum_h g_h\otimes [\omega]^{n/2}\in \tcf^*(S^0)\otimes_{\mathbb{Z}}H^*(X; \mathbb{C})\cong \tcf^*(X)\otimes_{\mathbb{Z}}\mathbb{C}$. Formula \eqref{product formula 200} yields 
\begin{align*}
 \langle f,g\rangle_X 
  &\equiv \sum_h \langle f_h,g_h\rangle \cdot (x_h \cup [\omega]^{n/2}) \neq 0,\end{align*}
since each $\langle f_h,g_h\rangle$ is nonzero and since the elements $\{ x_h \cup [\omega]^{n/2}\}\subseteq H^{2n}(X;\mathbb{C})$ are $\mathbb{C}$-linearly independent due to the hard Lefschetz theorem.
\end{proof}

\bibliographystyle{alpha}

\begin{thebibliography}{Mov15}

\bibitem[Arl92]{MR1193149}
Dominique Arlettaz.
\newblock The order of the differentials in the {A}tiyah-{H}irzebruch spectral
  sequence.
\newblock {\em $K$-Theory}, 6(4):347--361, 1992.

\bibitem[Bak90]{MR1037690}
Andrew Baker.
\newblock Hecke operators as operations in elliptic cohomology.
\newblock {\em J. Pure Appl. Algebra}, 63(1):1--11, 1990.

\bibitem[CS22]{top-hecke}
L.~Candelori and A.~Salch.
\newblock Topological {H}ecke eigenforms, 2022.
\newblock available on arXiv.

\bibitem[Dol66]{MR0198464}
Albrecht Dold.
\newblock {\em Halbexakte {H}omotopiefunktoren}, volume~12 of {\em Lecture
  Notes in Mathematics}.
\newblock Springer-Verlag, Berlin-New York, 1966.

\bibitem[Mov15]{MR3373249}
Hossein Movasati.
\newblock Quasi-modular forms attached to elliptic curves: {H}ecke operators.
\newblock {\em J. Number Theory}, 157:424--441, 2015.

\end{thebibliography}

\def\cprime{$'$} \def\cprime{$'$} \def\cprime{$'$} \def\cprime{$'$}

\end{document}